\documentclass[smallextended]{svjour3}     
\smartqed  
\usepackage{graphicx}
\usepackage{amssymb,amsfonts}
\usepackage{amsmath}
\usepackage{algorithm}
\usepackage{algorithmic}
\usepackage{latexsym}
\numberwithin{theorem}{section}
\numberwithin{definition}{section}
\numberwithin{proposition}{section}
\numberwithin{lemma}{section}
\numberwithin{remark}{section}
\numberwithin{example}{section}
\numberwithin{algorithm}{section}

\begin{document}

\title{Generalized Seikkala Derivatives and their application for solving Fuzzy Wave Equation}

\titlerunning{}        
\author{U. M. Pirzada}


\author{U. M. Pirzada         \and
        Raju K. George 
}

\institute{U. M. Pirzada \at
              School of Engineering and Technology, Navrachana University, Vadodara 391410, Gujarat, India. \\
              \email{salmapirzada@yahoo.com}           
           \and
          Raju K. George \at
              Department of Mathematics, Indian Institute of Space Science and Technology, thiruvananthapuram, Kerala 695 547, India.\\
	      \email{rkg.iist@gmail.com}  
}

\date{}

\maketitle

\begin{abstract}
This paper presents the new generalized Seikkala derivatives (gS- derivatives) of fuzzy-valued functions. The solution of fuzzy wave equation is proposed and analyzed using gS-derivatives whose crisp solution is expressed in terms of Fourier series. 
\end{abstract}

\section{Introduction}

The fuzzy partial differential equation (FPDE) means the generalization of partial differential equation (PDE) in fuzzy sense. While doing modeling of real situation in terms of partial differential equation, we see that the variables and parameters involve in the equations are uncertain (in the sense that they are not completely known or inexact or imprecise). Often common initial or boundary condition of ambient temperature is a fuzzy condition since ambient temperature is prone to variation in a range. We express this impreciseness and uncertainties in terms of fuzzy numbers. So we come across with fuzzy partial differential equations . In \cite{BU}, Buckley and Feuring (1999) proposed a procedure to examine solutions of elementary fuzzy partial differential equations. First they verified the Buckley - Feuring (BF) solution exist or not. If the BF-solution fails to exist they looked for the Seikkala solution. The solutions are based on Seikkala derivative introduced in \cite{SE}. Their proposed method works for elementary fuzzy partial differential equations. They assumed the solution of FPDE is not defined in terms of series. \\

\subsection{Brief literature survey}

In \cite{AL}, Allahviranloo (2002) proposed difference method to solve FPDEs. This method was based on Seikkala derivative of fuzzy functions. The Adomian method was studied to find the approximate solution of fuzzy heat equation in \cite{AL1}(2009). While in \cite{AL2}, Allahviranloo and Afshar (2010) presented numerical methods for solving the fuzzy partial differential equations. These numerical methods were based on the derivative due to Bede and Gal \cite{BE}. Mahmoud and Iman \cite{MA} (2013) presented finite volume method that solves some FPDEs such as fuzzy hyperbolic equations, fuzzy parabolic equations and fuzzy elliptic equations. They have obtained explicit, implicit and Crank-Nicolson schemes for solving fuzzy heat equation. Study of heat, wave and Poisson equations with uncertain parameters are given in \cite{BE1} (2013). Allahviranloo et al. have studied fuzzy solutions for fuzzy heat equation with fuzzy initial value based on generalized Hukuhara differentiability in \cite{AL3}(2014). Pirzada  and Vakaskar have discussed the solution of fuzzy heat equations using Adomian Decomposition in \cite{PI1} (2015). Solution of fuzzy heat equation under
fuzzified thermal diffusivity is discussed by Pirzada and Vakaskar in \cite{PI2} (2017). Fuzzy solution of homogeneous heat equation having solution in Fourier series form is analyized in \cite{PI3}(2019).

Applications to FPDEs are presented with a new inference method in \cite{CH}(2009). B.A. Faybishenko \cite{FA} presented a hydrogeologic system as a fuzzy system in (2004). He derived a fuzzy logic form of parabolic-type partial differential equation and solved using basic principles of fuzzy arithmetic. The exact solutions of fuzzy wave-like equations with variable coefficients by a variational iteration method is proposed in \cite{AL4} (2011). Series solution of fuzzy wave-like equations with variable coefficients were presented in \cite{HA}(2013). Biswas and Roy have defined generalization of Seikkala derivative and solved fuzzy Volterra integro-differential equations using differential transform method in \cite{BI} (2018).

\subsection{Motivation and novelty of the proposed work}

Limitations in other fuzzy derivatives:
\begin{itemize}
\item In the solution of fuzzy differential equation, we need differentiability of level functions of fuzzy-valued function. In Sekkala derivatives, level functions are differentiable but condition $f_{1}^{\prime}(\alpha) \leq f_{2}^{\prime}(\alpha)$, for all $\alpha$ may not satisfied for many fuzzy-valued functions.
\item Hukuhara derivatives are based on Hukuhara difference which exists in very restrictive situation.
\item Generalized H-derivatives are less restrictive but in this differentiability, the level functions may not be differentiable.
\end{itemize}

Another motivation for the current study is the following: \\

Buckley and Feuring \cite{BU} have introduced BF-solution of non-homogeneous elementary fuzzy partial differential equations in the form $\phi(D_{t},D_{x})\tilde{U} = \tilde{F}(x,t, \tilde{K})$. If we consider a homogeneous fuzzy partial differential equation, i.e. $\tilde{F}(x, t, \tilde{K}) = \tilde{0}$ then we can not apply sufficient condition to find a BF-solution. \\

With above motivations, we propose the new generalized Seikkala derivative of fuzzy-valued function which is appropriate for solution of fuzzy differential equations and it is less restrictive. Moreover, we find solutions of second order homogeneous fuzzy wave equation based on Seikkala solution approach. Using generalized Seikkala derivative, we solve fuzzy wave equation with specific fuzzy boundary and initial conditions whose crisp solution is expressed in Fourier series. \\

The paper is organized as follows. \\
The basic concepts of fuzzy numbers are given in Sec. 2. The generalized Seikkala derivative (gS-derivatives) of fuzzy-valued function is proposed in Sec. 3. Properties, relation between gS-derivative and other derivatives are discussed in the same section. The generalized Seikkala partial derivatives are also proposed in the section. Sec. 4 deals with the solution of fuzzy wave equation with specific fuzzy boundary and initial conditions. Analysis of solution based on Fourier series is given in Sec. 5. We conclude in the last Section. 


\section{Fuzzy numbers and arithmetics}
We start with some basic definitions.
\label{sec:2}
\begin{definition}\label{def1} A fuzzy set $\tilde{a}$ with membership function $\tilde{a}:\mathbb{R} \to[0,1]$, where $\mathbb{R}$ is the set of real numbers, is called a fuzzy number if it is normal, upper semi-continuous, quasi-concave function and closure of the set $\{x \in \mathbb{R} / \tilde{a}(x) >0\}$ is compact. The set of all fuzzy numbers on $\mathbb{R}$ is denoted by $F(\mathbb{R})$. 
\end{definition}

\begin{definition}
For all $\alpha \in (0,1]$, $\alpha$-level set $\tilde{a}_{\alpha}$ of any $\tilde{a}\in F(\mathbb{R})$ is defined as 
\begin{eqnarray*}
\tilde{a}_{\alpha} = \{ x \in \mathbb{R}/ \tilde{a}(x)\geq \alpha \} .
\end{eqnarray*} 
The 0-level set $\tilde{a}_{0}$ is defined as the closure of the set 
\begin{eqnarray*}
\{x \in \mathbb{R} / \tilde{a}(x) >0\}.
\end{eqnarray*} 
\end{definition}

The following Theorem of Goetschel and Voxman \cite{GO}, shows the characterization of a fuzzy number in terms of its $\alpha$-level sets.
\begin{theorem}\label{thm21}
For $\tilde{a} \in F(\mathbb{R})$, define two functions ${a}_{1}(\alpha),{a}_{2}(\alpha): [0,1] \to \mathbb{R} $. Then
\begin{enumerate}
\item [(i)] {${a}_{1} (\alpha)$ is bounded left continuous non-decreasing function on (0,1];}
\item [(ii)] {${a}_{2}(\alpha)$ is bounded left continuous non-increasing function on (0,1];}
\item [(iii)] {${a}_{1}(\alpha)$ and ${a}_{2}(\alpha)$ are right continuous at $\alpha = 0$;}
\item [(iv)] {${a}_{1}(\alpha) \leq {a}_{2}(\alpha)$.}
\end{enumerate}
Moreover, if the pair of functions ${a}_{1}(\alpha)$ and ${a}_{2}(\alpha)$ satisfy the conditions (i)-(iv), then there exists a unique $\tilde{a} \in F(\mathbb{R})$ such that $\tilde{a}_{\alpha} = [{a}_{1}(\alpha), {a}_{2}(\alpha)]$, for each $\alpha \in [0,1]$.
\end{theorem}
%
%
\begin{definition}\label{def3} According to Zadeh's extension principle, scalar multiplication of fuzzy number $\tilde{a}$ with a scalar $\lambda \in \mathbb{R}$ by its $\alpha$-level sets is defined as follows:
\begin{eqnarray*}
(\lambda \odot \tilde{a})_{\alpha} & = &
[\lambda\cdot {a}_{1}(\alpha),\lambda\cdot {a}_{2}(\alpha)],~if~\lambda \geq 0 \\
			           & = &
[\lambda\cdot {a}_{2}(\alpha),\lambda\cdot {a}_{1}(\alpha)],~if~\lambda < 0, 
\end{eqnarray*}
where $\alpha$-level sets of $\tilde{a}$ is $\tilde{a}_{\alpha} = [{a}_{1}(\alpha), {a}_{2}(\alpha)]$, for $\alpha \in [0,1]$. 
\end{definition}
The fuzzy-valued function is defined as follows:
\begin{definition}  \label{def4}
A function $\tilde{f}: V \to F(\mathbb{R})$ is called a fuzzy-valued function, where $V$ is a real vector space. That is, for each $x \in V$, $\tilde{f}(x)$ is a fuzzy number. Corresponding to $\tilde{f}$ and $\alpha \in [0,1]$, we denote two real-valued functions ${f}_{1}(x,\alpha)$ and ${f }_{2}(x,\alpha)$ on $V$ for all $x \in V$. These functions ${f}_{1}(x,\alpha)$ and ${f }_{2}(x,\alpha)$ are called lower and upper $\alpha$-level functions of $\tilde{f}$, respectively.
\end{definition}

\section{Generalized Seikkala Derivatives}

Seikkala derivative of fuzzy-valued function is defined as follows. The definition is adopted from Seikkala (1987)\cite{SE}.

\begin{definition} \label{def3.1}
Let $I$ be subset of $\mathbb{R}$ and $\tilde{y}$ be a fuzzy-valued function defined on $I$. The $\alpha$-level sets $\tilde{y}_{\alpha}(t) = [y_{1}(t, \alpha), y_{2}(t, \alpha)]$ for $\alpha \in [0,1]$ and $t \in I$. We assume that derivatives of $y_{i}(t, \alpha)$, $i = 1, 2$ exist for all $t \in I$ and for each $\alpha$. 
We define $(\tilde{y}^{\prime}(t))_{\alpha} = [y_{1}^{\prime}(t, \alpha), y_{2}^{\prime}(t, \alpha)]$ for all $t \in I$, all $\alpha$. 

If, for each fixed $t \in I$, $(\tilde{y}^{\prime}(t))_{\alpha}$ defines the $\alpha$-level set of a fuzzy number, then we say that Seikkala derivative of $\tilde{y}(t)$ exists at $t$ and it is denoted by fuzzy-valued function $\tilde{y}^{\prime}(t)$.
\end{definition}

The Seikkala derivative involves two steps:
\begin{enumerate}
\item [(1)] Check both level functions are differentiable or not
\item [(2)] Check level sets of derivatives define fuzzy numbers or not.
\end{enumerate}

Sufficient conditions for $(\tilde{y}^{\prime}(t))_{\alpha}$ to define $\alpha$-level sets of a fuzzy number are \cite{BU}:
\begin{enumerate}
\item [(i)] $y_{1}^{\prime}(t, \alpha)$ is an increasing function of $\alpha$ for each $t \in I$;
\item [(ii)] $y_{2}^{\prime}(t, \alpha)$ is an decreasing function of $\alpha$ for each $t \in I$;
\item [(iii)] $y_{1}^{\prime}(t, 1) \leq y_{2}^{\prime}(t, 1)$ for all $t \in I$.
\end{enumerate}

There are certain functions which exist in real situation but their Seikkala derivatives do not exist. We consider two such examples.

\begin{example}\label{ex3.1}
Consider a fuzzy-valued function $\tilde{g}(t) = \tilde{a} \odot \exp(-t)$, $t \in \mathbb{R}$ and $\tilde{a}$ is a fuzzy number with $\alpha$-level sets

\[(\tilde{g}(t))_{\alpha} = [g_{1}(t, \alpha), g_{2}(t, \alpha)] = [a_{1}(\alpha) \exp(-t), a_{2}(\alpha) \exp(-t)].\]

To check Seikkala differentiability of given fuzzy-valued function, first we check both its level functions are differentiable or not. 
We see that $g_{1}(t, \alpha)$ and $g_{2}(t, \alpha)$ are differentiable for $t \in \mathbb{R}$. Next, we check that the level sets 

\[(\tilde{g}^{\prime}(t))_{\alpha} = [g_{1}^{\prime}(t, \alpha), g_{2}^{\prime}(t, \alpha)] = [-a_{1}(\alpha)\exp(-t), -a_{2}(\alpha)\exp(-t)]\]
define a fuzzy number for each $t \in \mathbb{R}$. By checking sufficient conditions for $(\tilde{g}^{\prime}(t))_{\alpha}$ to define $\alpha$-level sets of fuzzy number,
\begin{enumerate} 
\item [(i)] $g_{1}^{\prime}(t, \alpha)$ is an increasing function of $\alpha$ for each $t \in \mathbb{R}$;
\item [(ii)] $g_{2}^{\prime}(t, \alpha)$ is a decreasing function of $\alpha$ for each $t \in \mathbb{R}$; and
\item [(iii)] $g_{1}^{\prime}(t, 1) \leq g_{2}^{\prime}(t, 1)$, for all $t \in \mathbb{R}$,
\end{enumerate} 
we see that  

\[{{\partial g_{1}^{\prime}(t, \alpha)} / {\partial \alpha}} = -a_{1}^{\prime}(\alpha) \exp(-t) < 0\] as $a_{1}^{\prime}(\alpha) > 0$ and 

\[{{\partial g_{2}^{\prime}(t, \alpha)}/ {\partial \alpha}} = -a_{2}^{\prime}(\alpha) \exp(-t) > 0\] as $a_{2}^{\prime}(\alpha) < 0$. Therefore $g_{1}^{\prime}(t, \alpha)$ is not an increasing function and $g_{2}^{\prime}(t, \alpha)$ is not a decreasing function. Hence, Seikkala derivative of $\tilde{g}$ does not exist.
\end{example}

We consider another example of fuzzy function which occur in uncertain periodic motion of an object whose Seikkala derivative does not exist.

\begin{example}\label{ex3.2}
Consider a fuzzy-valued function $\tilde{h}(t) = \tilde{a} \odot \sin(t)$, $t \in [0, \pi]$, where $\tilde{a}$ is a fuzzy number. The $\alpha$- level sets of $\tilde{h}(t)$ are $[a_{1}(\alpha) \sin(t), a_{2}(\alpha) \sin(t)]$. The level functions are differentiable but their derivatives $h_{1}^{\prime}(t, \alpha) = a_{1}(\alpha) \cos(t)$ and $h_{2}^{\prime}(t, \alpha) = a_{2}(\alpha) \cos(t)$ do not define fuzzy number for each $t \in [\pi/2, \pi]$ and hence $\tilde{h}$ is not Seikkala differentiable for $t \in [\pi/2, \pi]$.
\end{example}

To overcome this limitation, we define the generalized Seikkala derivative (gS-derivative) of a fuzzy-valued function

\begin{definition}\label{def3.2}
Let $I$ be a real interval. A fuzzy-valued function $\tilde{f}: I \to F(\mathbb{R})$ with $\alpha$-level sets

\[(\tilde{f}(t))_{\alpha} = [ f_{1}(t, \alpha),  f_{2}(t, \alpha)],\]

for $t \in I$ and $ \alpha \in [0, 1]$, is said to have generalized Seikkala derivative $\tilde{f}^{\prime}(t)$ if  $f_{1}(t, \alpha)$ and  $f_{2}(t, \alpha)$ are differentiable for each $t \in I$ and 
\[(f^{\prime}(t))_{\alpha} = [ \min \{f_{1}^{\prime}(t, \alpha), f_{2}^{\prime}(t, \alpha)\}, \max\{ f_{1}^{\prime}(t, \alpha), f_{2}^{\prime}(t, \alpha)\}], \]
for all $\alpha$ defines a fuzzy number for each $t \in I$. 
\end{definition}

Biswas and Roy have defined generalization of Seikkala derivative in \cite{BI} (2018). 

\begin{definition}\label{def3.3}
Let $\tilde{f} : (a, b) \to F(\mathbb{R})$ and $t_{0} \in (a, b)$. Then the generalized Seikkala derivative (gS-derivative) of $\tilde{f}(t)$ at $t_{0}$ is denoted $\tilde{f}^{\prime}(t_{0})$ and defined by
\begin{enumerate}
\item [(i)] if $f_{1}^{\prime}(t_{0}, \alpha)$, $f_{2}^{\prime}(t_{0}, \alpha)$ exist and $f_{1}^{\prime}(t_{0}, \alpha) \leq f_{2}^{\prime}(t_{0}, \alpha)$ then
\[
f_{\alpha}^{\prime}(t_{0}) = [f_{1}^{\prime}(t_{0}, \alpha), f_{2}^{\prime}(t_{0}, \alpha)]
\]
\item [(ii)] if $f_{1}^{\prime}(t_{0}, \alpha)$, $f_{2}^{\prime}(t_{0}, \alpha)$ exist and $f_{1}^{\prime}(t_{0}, \alpha) \geq f_{2}^{\prime}(t_{0}, \alpha)$ then
\[
f_{\alpha}^{\prime}(t_{0}) = [f_{2}^{\prime}(t_{0}, \alpha), f_{1}^{\prime}(t_{0}, \alpha)]
\]
\end{enumerate}
\end{definition}
The relation between Definition \ref{def3.2} and \ref{def3.3} of generalized Seikkala differentiability is given below.
\begin{theorem}
Definition \ref{def3.2} and \ref{def3.3} are equivalent.
\end{theorem}
\begin{proof}
The proof is straight forward and therefore omitted. \qed
\end{proof}

\begin{theorem}
If $\tilde{f}$ is Seikkala differentiable by Definition \ref{def3.1} then it is gS-differentiable by Definition \ref{def3.2} and \ref{def3.3}.
\end{theorem}
\begin{proof}
Since $\tilde{f}$ is S-differentiable, by definition, derivatives $f_{1}^{\prime}(t, \alpha)$ and  $f_{2}^{\prime}(t, \alpha)$ exist and the set $(\tilde{f}^{\prime}(t))_{\alpha}$ defines a fuzzy number for each $t$ in domain. 
\[
(\tilde{f}^{\prime}(t))_{\alpha} = [f_{1}^{\prime}(t, \alpha), f_{2}^{\prime}(t,\alpha)],
\]
for all $\alpha \in [0,1]$. It satisfied the Definition \ref{def3.3} as $ f_{1}^{\prime}(t, \alpha) \leq f_{2}^{\prime}(t, \alpha)$. Now we write the above equation in following way
\[
(\tilde{f}^{\prime}(t))_{\alpha} = [\min\{f_{1}^{\prime}(t, \alpha), f_{2}^{\prime}(t,\alpha)\}, \max\{f_{1}^{\prime}(t, \alpha), f_{2}^{\prime}(t,\alpha)\}],
\]
for all $\alpha \in [0,1]$. Therefore, $\tilde{f}$ is gS-differentiable by Definition \ref{def3.2}. \qed
\end{proof}

\begin{remark}
If $\tilde{f}$ is gS-differentiable by Definition \ref{def3.2} then it may not be S-differentiable by Definition \ref{def3.1}. For instance, fuzzy-valued function $\tilde{g}(t) = \tilde{a} \odot \exp(-t)$, $t \in \mathbb{R}$ in Example \ref{ex3.1} is not S-differentiable. The following example shows that it is gS-differentiable by Definition \ref{def3.2}.
\end{remark}

We see that the uncertain functions defined in above examples are gS-differentiable.

\begin{example} \label{ex3.3}
Consider the fuzzy-valued function $\tilde{g}(t) = \tilde{a} \odot \exp(-t)$, $t \in \mathbb{R}$, defined in Example \ref{ex3.1}. The derivatives of level functions of $\tilde{g}(t)$ are 

\[ g_{1}^{\prime}(t, \alpha) = -a_{1}(\alpha) \exp(-t)\] and 
\[ g_{2}^{\prime}(t, \alpha) = -a_{2}(\alpha) \exp(-t).\]

By definition of gS-differentiability, $\alpha$-level sets $(\tilde{g}^{\prime}(t))_{\alpha}$ defined as

\[(\tilde{g}^{\prime}(t))_{\alpha} = [\min\{-a_{1}(\alpha) \exp(-t), -a_{2}(\alpha) \exp(-t)\}, \max\{-a_{1}(\alpha) \exp(-t), -a_{2}(\alpha) \exp(-t)\}] \]

which is equal to 

\[ (\tilde{g}^{\prime}(t))_{\alpha} = [-a_{2}(\alpha) \exp(-t), -a_{1}(\alpha) \exp(-t)]\] \\

as $- a_{2}(\alpha) \leq - a_{1}(\alpha)$ and $\exp(-t) \geq 0$, for all $t$. Therefore, $\tilde{g}$ is gS-differentiable with derivative $\tilde{g}^{\prime}(t) = -\tilde{a} \odot \exp(-t)$. 
\end{example}

\begin{example}\label{ex3.4}
The fuzzy-valued function $\tilde{h}(t)$ defined in Example \ref{ex3.2} is gS-differentiable with derivative $\tilde{h}^{\prime}(t) = \tilde{a} \odot \cos(t)$. The $\alpha$-level sets of $\tilde{h}^{\prime}(t)$ are $[a_{1}(\alpha) \cos(t), a_{2}(\alpha) \cos(t)]$ for $t \in [0, \pi/2]$ and $[a_{2}(\alpha) \cos(t), a_{1}(\alpha) \cos(t)]$ for $t \in (\pi/2, \pi]$.
\end{example}

Now we define generalized Seikkala partial derivatives of two variables fuzzy-valued function $\tilde{f} : \mathbb{R}^{2} \to F(\mathbb{R})$.

\begin{definition}
Let $X$ be a subset of $\mathbb{R}^{2}$. A fuzzy-valued function $\tilde{f}: X \to F(\mathbb{R})$ with $\alpha$-level sets

\[(\tilde{f}(x, t))_{\alpha} = [ f_{1}(x, t, \alpha),  f_{2}(x,t, \alpha)],\]

for $(x, t) \in X$ and $ \alpha \in [0, 1]$, is said to have generalized Seikkala partial derivative ${\partial \tilde{f}} \over {\partial t}$ if  both partial derivatives $ {\partial f_{1}(x, t, \alpha)} \over {\partial t}$ and  ${\partial f_{2}(x, t, \alpha)} \over {\partial t}$ exist and continuous for each $(x, t) \in X$, for all $\alpha$ and 

\[\Big ({{\partial \tilde{f}} \over {\partial t}}\Big)_{\alpha} = \Big [ \min \Big \{ {{\partial f_{1}} \over {\partial t}}, {{\partial f_{2}} \over {\partial t}} \Big \}, \max \Big \{ {{\partial f_{1}} \over {\partial t}}, {{\partial f_{2}} \over {\partial t}} \Big \} \Big], \]
for all $\alpha$, defines a fuzzy number for each $(x, t) \in X$. In similar way, we can define generalized Seikkala partial derivative ${\partial \tilde{f}} \over {\partial x}$.

\end{definition}

\begin{example}
Consider a fuzzy-valued function $\tilde{f}: X \to F(\mathbb{R})$ by $\tilde{f}(t,x) = \tilde{c} \odot (e^{t}\sin{x})$, for $(x, t) \in \mathbb{R} \times [0, \pi]$. It is easily checked that the gS-partial derivative ${\partial \tilde{f}} \over {\partial x}$ of $\tilde{f}$ exists.
\end{example}

Now we define generalized Seikkala differentiability of fuzzy-valued function $\tilde{f}: X \subset \mathbb{R}^{2} \to F(\mathbb{R})$.

\begin{definition}
A fuzzy-valued function $\tilde{f}: X \to F(\mathbb{R})$ is said to be generalized Seikkala differentiable if both generalized Seikkala partial derivatives ${\partial \tilde{f}} \over {\partial x}$ and ${\partial \tilde{f}} \over {\partial t}$ exists and the fuzzy partial derivatives are continuous.
\end{definition}

The second order generalized Seikkala partial derivatives are defined as follows:

\begin{definition}
If a fuzzy-valued function $\tilde{f}: X \to F(\mathbb{R})$ is generalized Seikkala differentiable then its second order Seikkala partial derivative ${\partial^2 \tilde{f}} \over {\partial t^2}$ is exists if  both partial derivatives $ {\partial^2 f_{1}} \over {\partial t^2}$ and  ${\partial f_{2}} \over {\partial t^2}$ exist and continuous for each $(x, t) \in X$, for all $\alpha$ and 

\[\Big ({{\partial^2 \tilde{f}} \over {\partial t^2}}\Big)_{\alpha} = \Big [ \min \Big \{ {{\partial^2 f_{1}} \over {\partial t^2}}, {{\partial^2 f_{2}} \over {\partial t^2}} \Big \}, \max \Big \{ {{\partial^2 f_{1}} \over {\partial t^2}}, {{\partial^2 f_{2}} \over {\partial t^2}} \Big \} \Big], \]
for all $\alpha$, defines a fuzzy number for each $(x, t) \in X$. In similar way, we can define other second order generalized Seikkala partial derivatives.

\end{definition}

\section{Fuzzy wave equation}


\subsection{Fuzzy model} 

The form of second order partial differential equation extensively occurred is the wave equation. Examples include disruption of a body of fluid, vibration of string instruments, vibration of membrane and pressure perturbations in air. In these cases, if the amplitude of the disturbance is sufficiently small, the perturbation variable characterizing the disturbance will satisfy the wave equation. Under some physical sensible assumptions, an one-dimensional wave equation is derived as 
\begin{equation}
{{\partial^2 u}\over{\partial t^2}} = c^2 {{\partial^2 u}\over{\partial x^2}},
\end{equation}
where $c^2$ is called wave speed and $u(x,t)$ is the displacement of the string. These assumptions make parameters constant or precise. For instance, we assume that string is uniform. That is, mass per unit length is constant. To study the problem in real sense, we consider imprecise variables and parameters. Modeling the problem of wave equation with imprecise parameters and uncertain boundary and initial conditions lead to a fuzzy model of wave equation.  
\begin{equation}\label{eq2.1}
{{\partial^2 \tilde{u}}\over{\partial t^2}} = c^2 \odot {{\partial^2 \tilde{u}}\over{\partial x^2}},
\end{equation}
where $\tilde{u}(x,t)$ is the fuzzy displacement represented by fuzzy number at each $(x, t) \in [0, L_{1}] \times [0, L{2}]$, $ L_{1}, L_{2} > 0$. The fuzzy parameters involve in the boundary and initial conditions are also expressed as fuzzy numbers and ${{\partial^2 \tilde{u}}\over{\partial t^2}}$ and ${{\partial^2 \tilde{u}}\over{\partial x^2}}$ represent second order fuzzy partial derivatives of fuzzy-valued function $\tilde{u}(x,t)$.

\subsection{Solution}
Elementary fuzzy partial differential equations are studied by Buckley and Feuring in \cite{BU}. They have not considered the solution in Fourier series form. Some researchers have studied non-homogeneous fuzzy wave equation involving constant $K$. The solution of fuzzy wave equation fully depend on $K$. Here we consider the homogeneous fuzzy partial differential equation in the form of fuzzy wave equation does not involve constant $K$.  It is observed that BF-solution does not exist for the problem (\ref{eq2.1}), as we can not apply the condition for existence to find solution to this problem. We find the Seikkala solution (S-solution) of (\ref{eq2.1}) using generalized Seikkala partial derivatives.

\subsection{Seikkala solution approach}

The fuzzy function $\tilde{u}(x,t)$ is a S-solution of problem (\ref{eq2.1}) if generalized Seikkala partial derivatives exist and satisfy the equation. Let $\tilde{u}_{\alpha}(x,t) = [u_{1}(x,t,\alpha), u_{2}(x,t,\alpha)]$. We re-write the equation (\ref{eq2.1}) as system of crisp partial differential equations 

\begin{equation}\label{eq4}
{{\partial^2 u_1(x,t,\alpha)}\over{\partial t^2}} = c^2 {{\partial^2 u_{1}(x,t,\alpha)}\over{\partial x^2}}
\end{equation}

\begin{equation}\label{eq5}
{{\partial^2 u_2(x,t, \alpha)}\over{\partial t^2}} = c^2 {{\partial^2 u_{2}(x,t,\alpha)}\over{\partial x^2}}
\end{equation}

for all $(x,t) \in I_{1} \times I_{2}$ and all $\alpha \in [0,1]$. The fuzzy boundary conditions are $\tilde{u} (0,t) = \tilde{C}_{1}$ and $\tilde{u} (L,t) = \tilde{C}_{2}$, $t > 0$ and fuzzy initial conditions are $\tilde{u} (x,0) = \tilde{f}(x)$ and $\tilde{u}_{t}(x,0) = \tilde{g}(x)$, $0 < x < L$, where $\tilde{C}_{1}$, $\tilde{C}_{2}$ are fuzzy numbers and $\tilde{f}(x)$, $\tilde{g}(x)$ are fuzzy-valued functions of $x$. We write boundary conditions in terms of $\alpha$-level sets as 

\begin{equation}\label{bc1}
 u_{1}(0,t,\alpha) = c_{11}(\alpha),~ u_{2}(0,t,\alpha) = c_{12}(\alpha),
\end{equation}

\begin{equation}\label{bc2}
 u_{1}(L,t,\alpha) = c_{21}(\alpha),~ u_{2}(L,t,\alpha) = c_{22}(\alpha).
\end{equation}

The initial conditions 

\begin{equation}\label{ic1}
 u_{1}(x,0, \alpha) = {f}_{1}(x, \alpha),~ u_{2}(x,0, \alpha) = {f}_{2}(x, \alpha),
\end{equation}

\begin{equation}\label{ic2}
 {{\partial u_{1}(x,0, \alpha)} \over {\partial t}} = {g}_{1}(x, \alpha),~ {{\partial u_{2}(x,0, \alpha)} \over {\partial t}} = {g}_{2}(x, \alpha).
\end{equation}

Let $u_{i}(x,t, \alpha)$ solve equations (\ref{eq4}) and (\ref{eq5}) with boundary conditions (\ref{bc1}) and (\ref{bc2}) and initial conditions (\ref{ic1}) and (\ref{ic2}) , $i = 1,2$. If

\begin{equation}
[u_{1}(x,t, \alpha), u_{2}(x,t, \alpha)]
\end{equation}
 defines the $\alpha$-level set of a fuzzy number, for each $(x,t) \in I_{1} \times I_{2}$, then $\tilde{u}(x,t)$ is the S-solution.

\begin{remark}
The non-homogeneous elementary fuzzy partial differential equations which are solved in \cite{BU} using BF- solution and Seikkala solution (S-solution) are also solved using gS-derivatives and S-solution approach. Because, fuzzy-valued functions which are Seikkala differentiable are also generalized Seikkala differentiable and therefore S-solution exists.
\end{remark}


\section{Solution of a fuzzy wave equation involves Fourier series}
In this section, we find the solution of a fuzzy wave equation whose crisp solution is expressed in terms of Fourier series. Consider the fuzzy wave equation given in (\ref{eq2.1}) with fuzzy boundary conditions $\tilde{u}(0,t) = \tilde{u}(1,t) = \tilde{0} $ and fuzzy initial conditions $\tilde{u}(x,0) = \tilde{U_{0}}$, where $\tilde{U_{0}}$ is a fuzzy number, $\tilde{u}_{t}(x,0) = \tilde{0}$. First, we solve equations (\ref{eq4}) and (\ref{eq5}) subject to conditions

\begin{equation}
u_{i}(0,t,\alpha) = u_{i}(1,t, \alpha) = 0
\end{equation}
\begin{equation}
u_{i}(x,0,\alpha) = U_{0i}(\alpha), {{\partial u_{i}(x,0, \alpha)} \over {\partial t}} = 0,
\end{equation}
for $i = 1,2$. The solution is 
\begin{equation}\label{sol1}
u_{i} (x,t, \alpha) = U_{0i}(\alpha) {{4} \over {\pi}} \sum_{n = 0}^{\infty} {{\sin{((2n+1)x)}} \over {(2n+1)}} \cos{((2n + 1)t)},
\end{equation}
for $i = 1,2$. The solution (\ref{sol1}) is obtained using Separation of variables method.\\

If $[u_{1}(x,t, \alpha), u_{2}(x,t, \alpha)]$ defines $\alpha$-level sets of a fuzzy number for $x \in I_{1}$ and $t \in I_{2}$, then it is a fuzzy solution of (\ref{eq2.1}) with specified fuzzy boundary and initial conditions. Since $u_{i}(x,t, \alpha)$ are continuous and $u_{1}(x,t,1) = u_{2}(x,t,1)$, what we need to check is 
${{\partial u_{1}} \over {\partial \alpha}} \geq 0 $ and ${{\partial u_{2}} \over {\partial \alpha}} \leq 0 $. Since $\tilde{U_{0}}$ is a fuzzy number, we have $U_{01}^{\prime}(\alpha) > 0 $ and $U_{02}^{\prime}(\alpha) < 0 $ (by Theorem \ref{thm21} and assumption of continuity of $U_{01}(\alpha)$ and $U_{02}(\alpha)$). For the S-solution to exist

\begin{equation}
{{\partial u_{1}} \over {\partial \alpha}} = U_{01}^{\prime}(\alpha){{4} \over {\pi}} \sum_{n = 0}^{\infty} {\sin{((2n+1)x)} \over {(2n+1)}} \cos{((2n + 1)t)}
\end{equation}

should be non-negative, for all $\alpha \in [0,1]$ and 

\begin{equation}
{{\partial u_{2}} \over {\partial \alpha}} = U_{02}^{\prime}(\alpha){{4} \over {\pi}} \sum_{n = 0}^{\infty} {\sin{((2n+1)x)} \over {(2n+1)}} \cos{((2n + 1)t)}
\end{equation}

should be non-positive, for all $\alpha \in [0,1]$ and $(x,t) \in I_{1} \times I_{2}$. For that 
\begin{equation}\label{e1}
z(x,t) = {{4} \over {\pi}} \sum_{n = 0}^{\infty} {\sin{((2n+1)x)} \over {(2n+1)}} \cos{((2n + 1)t)}  
\end{equation} 
should be non-negative for $(x,t) \in I_{1} \times I_{2} $. 

\subsection{Analysis of results}

When $n = 0$, i.e., only first term in the series, we have

\[ z(x,t) = {{4} \over {\pi}} {sin{(x)}} \cos{(t)}  \]

is non-negative( see surface $z(x,t)$ in \textbf{Fig. 1}). Hence S-solution of the fuzzy wave equation exists for $x \in [0,\pi]$, $t \in [0,\pi/2]$ and it is given as 

\begin{equation}\label{solw0}
\tilde{u}(x,t) = \tilde{U}_{0} \odot {{4} \over {\pi}} {{\sin{(x)}} \cos{(t)}},
\end{equation}
as generalized Sekkala partial derivatives of $\tilde{u}(x,t)$ exist for $(x,t) \in [0, \pi] \times [0, \pi/2]$ where as Seikkala derivatives do not exist because the derivative of $\sin(x)$ is negative in the interval $[\pi /2, \pi]$ and derivative of $\cos(t)$ is also negative in $[0, \pi/ 2]$.

\begin{figure}[h]\label{fig1}
\begin{center}
\includegraphics[width=4.0in]{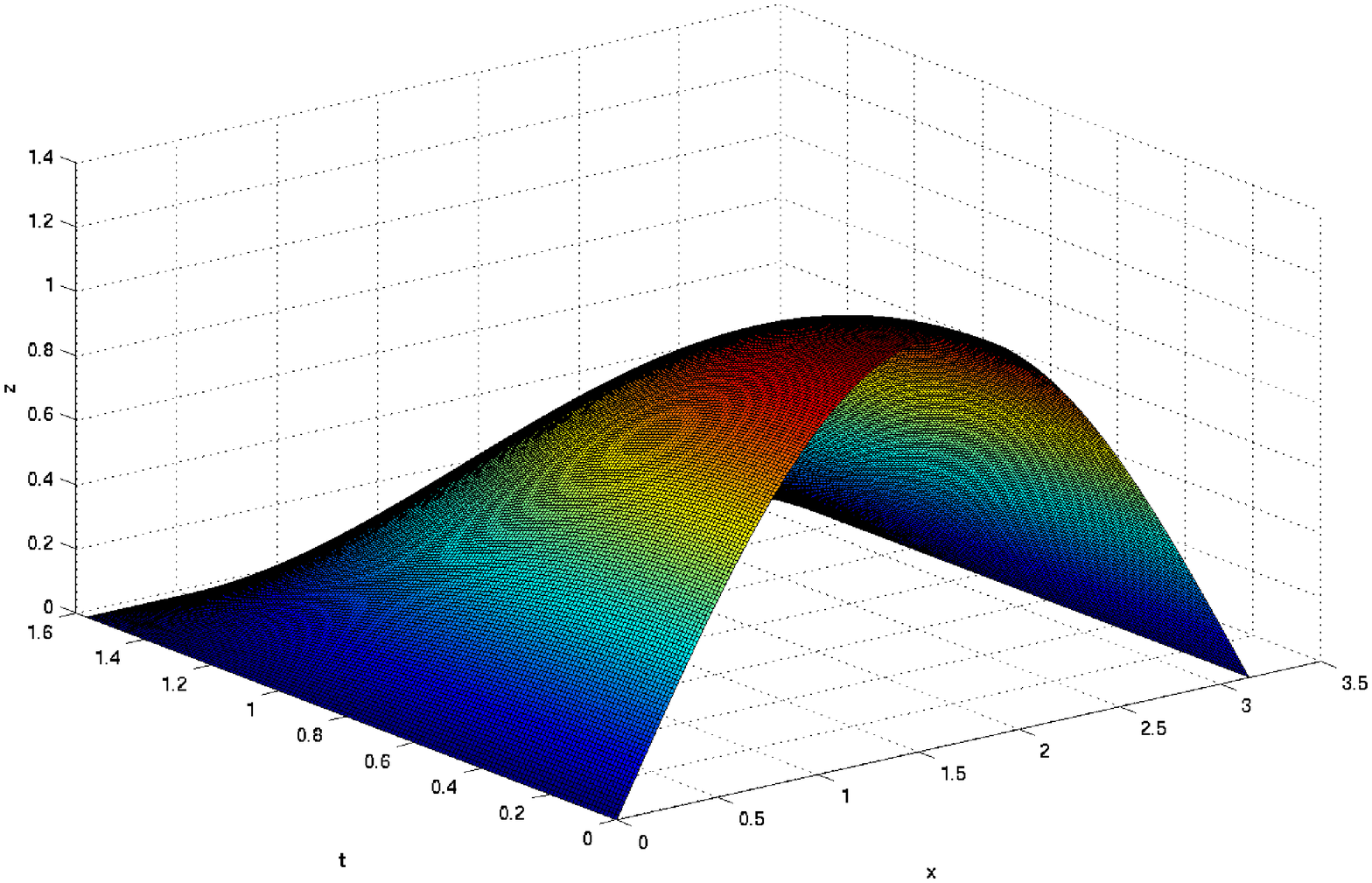}
\caption{$z(x,t) = {{4} \over {\pi}} sin{(x)} \cos{(t)}$}
\end{center}
\end{figure}
Now we take $n = 1$, we have 

\[
z(x,t) = {{4} \over {\pi}} \sum_{n = 0}^{1}{ {sin{((2n+1)x)} \cos{((2n + 1)t)} }  \over {(2n+1)}}
\] 
is non-negative for $x \in [0,0.78]$ and $t \in [0,0.78]$ (see \textbf{Fig. 2}). Therefore, S-solution exist in this domain and it as given as following

\begin{equation}\label{solw1}
\tilde{u}(x,t) = \tilde{U}_{0} \odot {{4} \over {\pi}} \sum_{n = 0}^{1} {{\sin{((2n+1)x)} \cos{((2n+1)t)}} \over {(2n + 1)}}.
\end{equation}
\begin{figure}[h]\label{fig2}
\begin{center}
\includegraphics[width=4.0in]{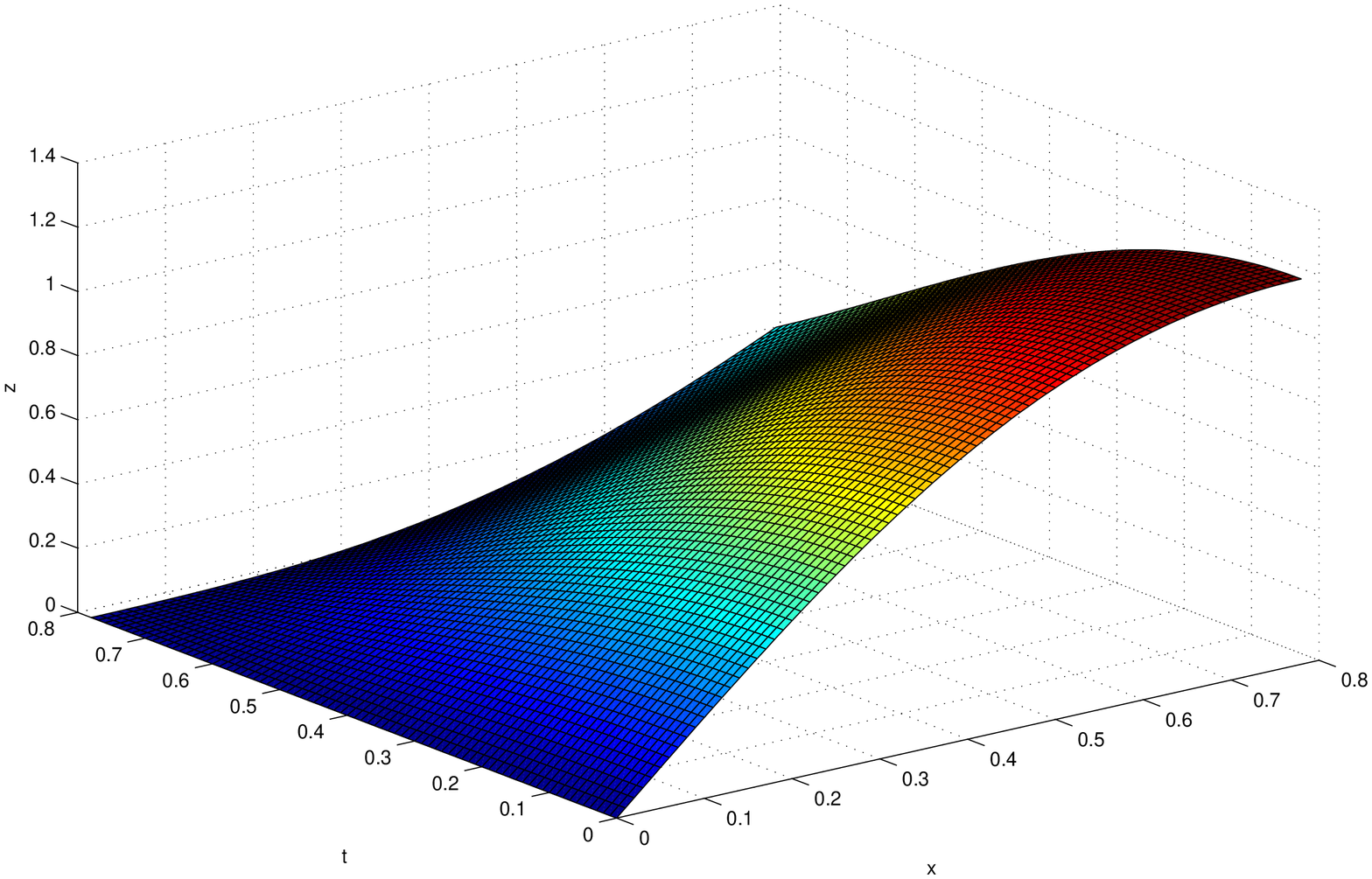}
\caption{$z(x,t) = {{4} \over {\pi}} \sum_{n=0}^{1}{{sin{((2n+1)x)}  \cos{((2n + 1)t)}}\over {(2n+1)}} $}
\end{center}
\end{figure}
For $n = 2$, we have 
\[z(x,t) = {{4} \over {\pi}} \sum_{n = 0}^{2}{ {sin{((2n+1)x)} \cos{((2n + 1)t)} }  \over {(2n+1)}}
\] is non-negative for $x \in [0,0.525]$, $t \in [0,0.525]$ (see \textbf{Fig. 3}). The S-solution is
\begin{equation}\label{solw2}
\tilde{u}(x,t) = \tilde{U}_{0} \odot {{4} \over {\pi}} \sum_{n = 0}^{2} {{\sin{((2n+1)x)} \cos{((2n+1)t)}} \over {(2n + 1)}}.
\end{equation}
\begin{figure}[h]\label{fig3}
\begin{center}
\includegraphics[width=4.0in]{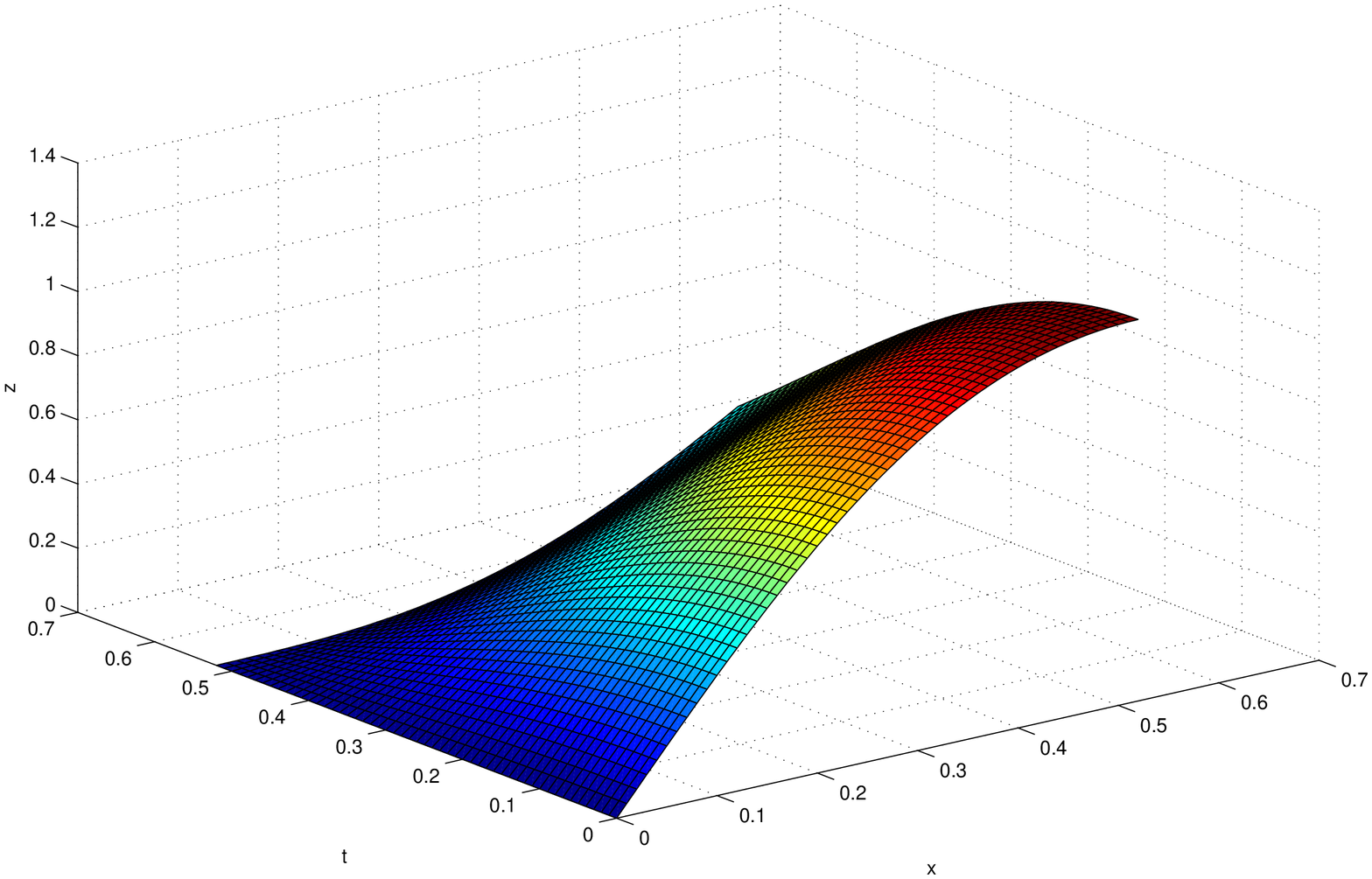}
\caption{$z(x,t) = {{4} \over {\pi}} \sum_{n=0}^{2}{sin{((2n+1)x)} \cos{((2n + 1)t)} \over {(2n+1)}} $}
\end{center}
\end{figure}
For $n = 3$, 
\[z(x,t) = {{4} \over {\pi}} \sum_{n = 0}^{3}{ {sin{((2n+1)x)} \cos{((2n + 1)t)} }  \over {(2n+1)}}
\]  
is non-negative for $x \in [0,0.39996]$, $t \in [0,0.39996]$ (see \textbf{Fig. 4}). S-solution exist in this domain.

\begin{equation}\label{solw3}
\tilde{u}(x,t) = \tilde{U}_{0} \odot {{4} \over {\pi}} \sum_{n = 0}^{3} {{\sin{((2n+1)x)} \cos{((2n+1)t)} \over {(2n + 1)}}}.
\end{equation}
 But if we increase $x$ and $t$ slightly, we have some negative values of $z(x,t)$. For instance, consider $x \in [0,0.4]$, $t \in [0,0.4]$ and see the surface in \textbf{Fig. 5}. Therefore, we S-solution of fuzzy wave equation exists in domain $x \in [0,0.39996]$, $t \in [0,0.39996]$.
\begin{figure}[h]\label{fig4}
\begin{center}
\includegraphics[width=4.0in]{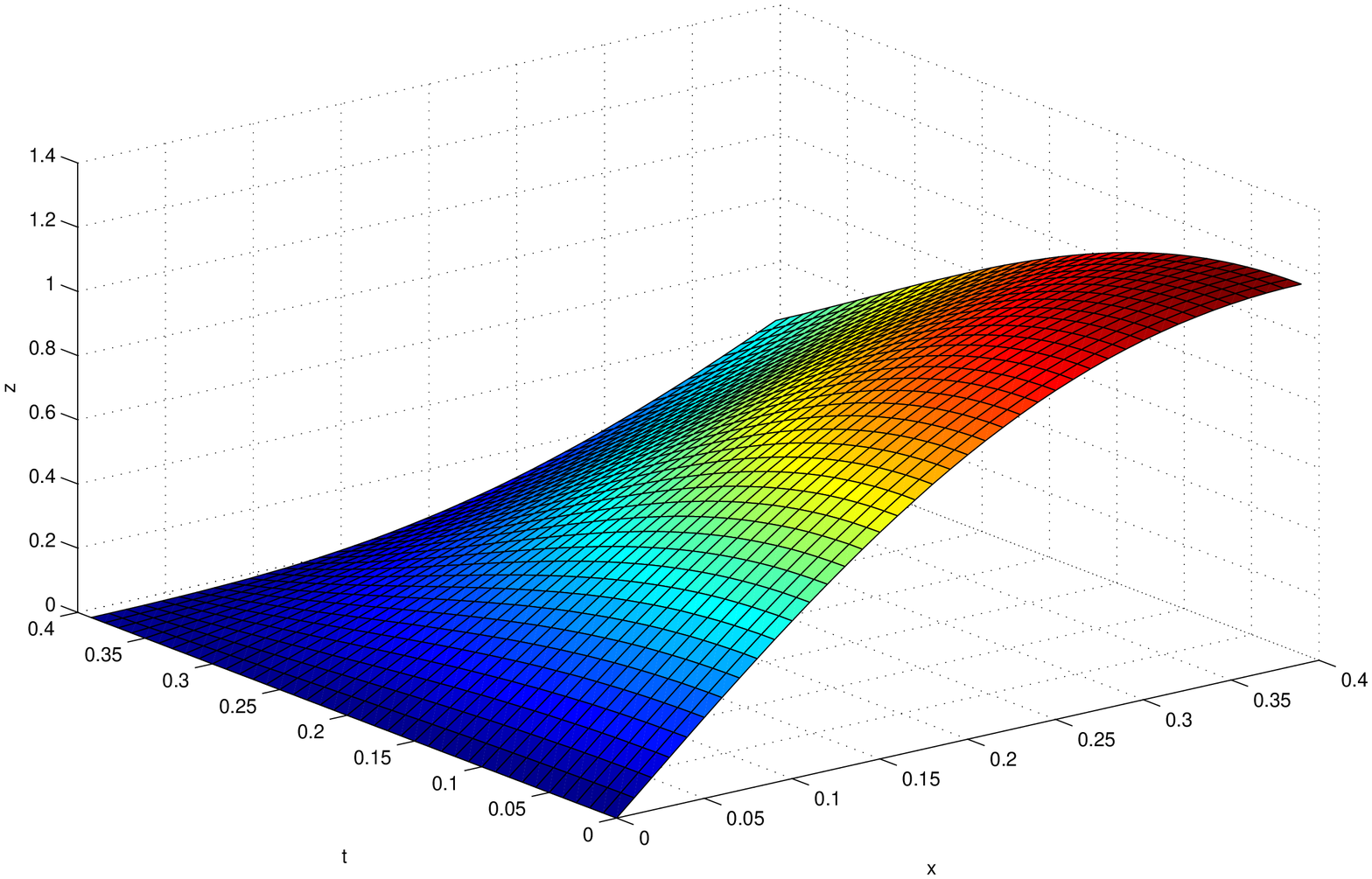}
\caption{$z(x,t) = {{4} \over {\pi}} \sum_{n=0}^{3}{sin{((2n+1)x)} \cos{((2n + 1)t)} \over {(2n+1)}} $}
\end{center}
\end{figure}

\begin{figure}[h]\label{fig5}
\begin{center}
\includegraphics[width=4.0in]{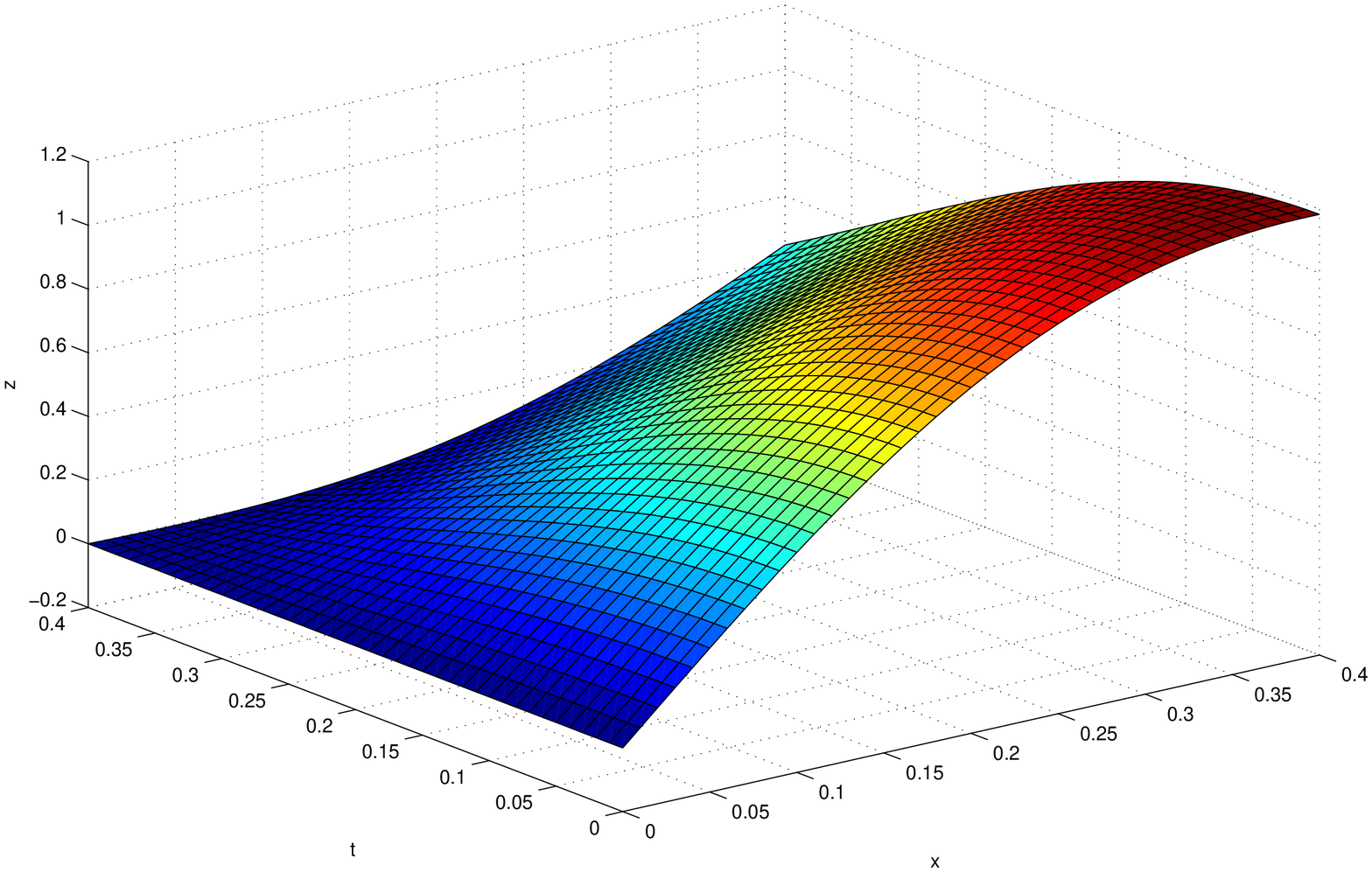}
\caption{$z(x,t) = {{4} \over {\pi}} \sum_{n=0}^{3}{sin{((2n+1)x)}\cos{((2n + 1)t)}  \over {(2n+1)}} $}
\end{center}
\end{figure}


\section{Conclusions}
We introduced new generalized Seikkala derivatives of fuzzy-valued function and studied the solution of fuzzy wave equation whose crisp solution involves Fourier series. We conclude that
\begin{itemize}
\item [(a)] A larger class of fuzzy-valued functions are generalized Seikkala differentiable.
\item [(b)] Homogeneous fuzzy partial differential equations can not solved using the Buckley- Feuring approach. Seikkala solution approach is applicable in this situation.
\item [(c)] Using Seikkala solution approach and generalized Seikkala partial derivatives, the solution of fuzzy wave equation is proposed whose crisp solution is expressed in terms of Fourier series.
\item [(d)] As we increase the number of terms of Fourier series in the fuzzy solution, the domain of fuzzy solution decreased.
\end{itemize}

\end{document}